\documentclass[a4paper,11pt]{article}
\usepackage[utf8x]{inputenc}
\usepackage{amssymb}
\usepackage{amsthm}
\usepackage{amsmath}
\usepackage[english]{babel}
\usepackage{accents}
\usepackage{amsfonts}
\usepackage{textcomp}
\usepackage{stmaryrd}
\usepackage{tikz}
\usepackage{booktabs}
\usepackage{url}
\usepackage[scaled=0.90]{helvet}
\usepackage{graphicx}

\usepackage{authblk}

\graphicspath{ {/home/stefan/Pictures/} }

\newtheorem{thm}{Theorem}
\newtheorem{Claim}[thm]{Claim}
\newtheorem{Fact}[thm]{Fact}

\newtheorem{Definition}[thm]{Definition}

\newcommand{\ZFC}{\mathsf{ZFC}}
\newcommand{\ZF}{\mathsf{ZF}}

\newcommand{\DC}{\mathsf{DC}}
\newcommand{\AD}{\mathsf{AD}}

\newcommand{\GCH}{\mathsf{GCH}}

\newcommand{\forceQ}{\mathbb{Q}}
\newcommand{\forceP}{\mathbb{P}}
\newcommand{\forceR}{\mathbb{R}}

\newcommand{\secret}[1]{}
\secret{Versteckter Text}

\title{A model with no ordinal-definable stationary, costationary subset of $\omega_1$}
\author{Stefan Hoffelner\footnote{The author was supported by FWF-GA\v CR grant no. 
17-33849L, Filters, ultrafilters and connections with forcing}}
\affil{Charles University Prague}
\begin{document}

\maketitle

\begin{abstract}
It is shown that the existence of a measurable cardinal is equiconsistent to 
a model of $\ZFC$ in which there is no ordinal-definable, stationary, costationary
subset of $\omega_1$
\end{abstract}

\section{Introduction}

The following short note has its origin in a question of B. Farkas who asked the author whether there is a model of 
$\ZFC$ in which every ordinal definable subset of $\omega_1$ is either club-containing or nonstationary.
It is a trivial observation that for any regular $\kappa > \omega_1$ there are always
ordinal-definable, stationary, co-stationary subsets of $\kappa$, e.g. the ordinals 
below $\kappa$ of some fixed cofinality $< \kappa$. This does not apply 
to $\omega_1$ of course, thus the question seems very natural.
 
By work of R. Solovay from the late 1960's (see \cite{Jech} Theorem 33.12), the axiom of determinacy implies that the 
club filter on $\omega_1$ is an ultrafilter. The axiom of determinacy is however 
inconsistent with the axiom of choice. M. Spector (\cite{Spector}) showed that $\ZF$ + ``the club filter on $\omega_1$ is an
ultrafilter'' is in fact equiconsistent with $\ZF$, but due to the lack of choice, the club filter is not $\sigma$-closed, 
even though the intersection of countably many club sets is still club. If one wants a model for
$\ZF$ + ``the club filter on $\omega_1$ is an ultrafilter'' + $\DC$, it is proved by W. Mitchell (see \cite{Mitchell})
that a measurable cardinal $\kappa$ with Michell rank $o(\kappa)= \kappa^{++}$ is an upper bound in terms of 
consistency strength. Note that all these results talk about models in which choice fails.

On the other hand R. Solovay (\cite{Solovay}) showed that any stationary set on a regular cardinal
$\kappa$ can be partitioned into $\kappa$-many stationary sets, its proof however 
relies heavily on the axiom of choice and does not allow in general any ordinal definable description
of this partition. It implies nevertheless that models of $\ZFC$ with a definable wellorder on $P(\omega_1)$
allow definable, stationary, costationary subsets of $\omega_1$.

Another observation uses Woodin's $\forceP_{max}$-forcing (see \cite{Larson} for details). If we assume that $\AD$ holds in $L(\forceR)$,
then, as already mentioned above, the club filter on $\omega_1$ is an ultrafilter. Forcing over $L(\forceR)$ with $\forceP_{max}$ produces a model $L(\forceR)[G]$ which
is in fact a model of $\ZFC$. Additionally $\forceP_{max}$ is $\sigma$-closed, and will not add 
new ordinal-definable subsets of $\omega_1$ (as it is weakly homogeneous, see the definition below),
thus $L(\forceR)[G]$ is a model of $\ZFC$ in which there is no ordinal-definable, stationary, co-stationary
set as all possible candidates would belong already to $L(\forceR)$ where the club filter is an ultrafilter.
However this argument uses the existence of infinitely many Woodin cardinals which is 
much more than actually needed.

We make use of two wellknown Inner models of $\ZFC$, $HOD$ and $L[U]$ for 
$U$ a normal measure on a measurable cardinal $\kappa$.
Recall that a set $X$ is ordinal definable if there is 
a formula $\varphi(v_0,...,v_n)$ and ordinals $\alpha_1,...,\alpha_n$ such that
$$x \in X \text{ iff } \varphi(x, \alpha_1,...,\alpha_n) \text{ is true. }$$
Then $HOD$, the class of hereditarily ordinal-definable sets is
defined to be the collection of all sets $X$ which are ordinal-definable
and all the elements of the transitive closure of $X$ are ordinal-definable as well.
It was G\"odel who showed that $HOD$ is a definable class (definable without parameters) and
$HOD$ satisfies $\ZFC$.

The second model is the canonical Inner model for one measurable
cardinal $L[U]$. Assume that $\kappa$ is a measurable
cardinal and let $U$ be a normal, nonprincipal, $\kappa$-complete ultrafilter on $\kappa$
witnessing the measurability of $\kappa$. 

\begin{thm}
  The inner model $L[U]$ has the following properties (due to Silver (see \cite{Silver}) and Kunen (\cite{Kunen}) respectively):
 \begin{itemize}
  \item in $L[U]$ the $\GCH$ does hold.
  \item there is exactly one normal, $\kappa$-complete ultrafilter on
  $\kappa$ namely $U \cap L[U]$. Consequentially $\kappa$ is measurable
  in $L[U]$ and $U$ is ordinal definable over $L[U]$.
  
 \end{itemize}

\end{thm}

Our notation will be completely standard. $\forceP$, $\forceQ$ and $\forceR$
will denote partial orders. Two partial orders $(\forceP, <_{\forceP})$ and
$(\forceQ, <_{\forceQ})$ are isomorphic if 
there is a bijection $\pi : \, \forceP \rightarrow \forceQ$ which 
respects $<_{\forceP}$, i.e. $p_1 <_{\forceP} p_2$ if and only if
$\pi(p_1) <_{\forceQ} \pi(p_2)$.
We start to recall a property which will be crucial for the rest of this paper.

\begin{Definition}
 A partial order $\forceP$ is weakly homogeneous if for every pair $p, q \in  \forceP$ there
 is an automorphism $\pi : \forceP \rightarrow \forceP$ such that $\pi(p)$ and $q$ are compatible.
\end{Definition}

The reason of its importance stems from the following.

\begin{Fact}
 Let $\forceP$ be a homogeneous partial order and $\varphi(v_0, v_1,...,v_n)$ 
 be a formula with free variables $v_0,...,v_n$. Then for any ground model sets 
 $x_0,..., x_n \in V$
 $$\Vdash_{\forceP} \varphi(x_0,..., x_n) \text{ or } \Vdash_{\forceP} \lnot \varphi (x_0,..., x_n)$$
 holds true.
\end{Fact}

Consequentially, if $\forceP$ is a homogeneous
notion of forcing with generic filter $G$ and $A$ is an ordinal-definable set of ordinals in the extension $V[G]$, then $A$ was already
an element of $V$. Indeed if $A$ is ordinal-definable over $V[G]$ using 
the formula $\varphi(v_0,...,v_n)$ then 
$\alpha_0 \in A$ iff $\varphi^{V[G]} (\alpha_0, \alpha_1,...,\alpha_n)$ iff 
there is a $p \in G \cap \forceP$ such that $p \Vdash_{\forceP} \varphi(\alpha_0, \alpha_1,...,\alpha_n)$
and by weak homogeneity this is equivalent to $1 \Vdash_{\forceP} \varphi(\alpha_0, \alpha_1,..., \alpha_n)$.
But the latter is an $A$-defining formula over $V$, thus $A \in V$ as desired.
If the forcing notion $\forceP$ is additionally assumed to be ordinal-definable itself, 
then an ordinal-definable subset $A$ of $V[G]$ is in fact ordinal-definable over the ground model $V$ for the
relation $\Vdash_{\forceP}$ and the weakest element $1_{\forceP}$ becomes ordinal-definable in that case.
As $\alpha_0 \in A$ iff $1 \Vdash_{\forceP} \varphi( \alpha_0,...,\alpha_n)$, 
the latter is an ordinal-definable formula in case that $\forceP$ is ordinal-definable, thus $A$, which was assumed
to be ordinal-definable over $V[G]$, is ordinal-definable
over $V$. 

We recall another standard fact, namely the club shooting forcing for a stationary $S \subset \omega_1$ which
adds a closed unbounded subset of $S$.
\begin{Definition}
 Let $S \subset \omega_1$ be a stationary subset. The club shooting forcing $\forceP_S$
 consists of closed, countable subsets of $S$ as conditions, ordered by end extension.
\end{Definition}
The definition of club shooting relies on the fact that 
stationary subsets $S$ of $\omega_1$ contain for every $\alpha < \omega_1$
a closed subset $C \subset S$ of ordertype $\alpha$. This 
statement becomes wrong for regular cardinals $\kappa > \aleph_1$. 
The next fact is a wellknown folklore result
and we state it without a proof:

\begin{Fact}
 For a stationary set $S$, the club shooting $\forceP_S$ is 
 $\omega$-distributive and weakly homogeneous.
\end{Fact}

\section{The proof}

Goal of this section is a proof of the following theorem.
\begin{thm}
 The existence of measurable cardinal is equiconsistent
 to the existence of a model of $\ZFC$ in which there is no
 ordinal-definable, stationary, costationary subset of $\omega_1$.
\end{thm}

\begin{proof}

One direction of the equiconsistency is observed easily.
Assume that $V$ is a model of $\ZFC$ in which there is no
ordinal-definable, stationary, costationary subset of $\omega_1$.
Consider the club filter of $V$, $U:= \{ X \subset \omega_1 \, : \, X \text{ contains a club} \}$.
It is clear that $U$ is ordinal-definable and hence $U \cap HOD$ is an element of $HOD$, the 
class of hereditarily ordinal definable subsets of $V$. As any $V$-ordinal-definable subset of
$\omega_1$ is either club containing or nonstationary we see that $U \cap HOD$ is
a nontrivial ultrafilter in $HOD$. As any subset $X \in HOD$ of $\omega_1$ can easily
be split into two proper subsets which are elements of $HOD$, it follows that
$U \cap HOD$ is also nonprincipal.
Finally, as the intersection of less than $\omega_1^V$-many club sets are club again,
we see that $U \cap HOD$ is an $\omega_1^V$-closed ultrafilter on $\omega_1^V$ in $HOD$, hence $U \cap HOD$ witnesses that $\omega_1^V$ is 
measurable in $HOD$.

The other direction is more involved and we start to prove it now. Assume there is a measurable
cardinal $\kappa$, and let the ultrafilter $U$ witness this. We let $V=L[U]$ be our ground model, 
which satisfies $\GCH$ and has $\kappa$ as a measurable cardinal, with $U$ being the unique
normal $\kappa$-closed ultrafilter witnessing the measurability of $\kappa$ in $L[U]$.
Note that $U$ is ordinal definable in $L[U]$ as it is the unique
ultrafilter on $\kappa$ which is $\kappa$-closed and normal.
Further note that $U$ is also ordinal definable in all set-sized 
generic forcing extensions $L[U][G]$ of $L[U]$, it is just
the unique set $X$ such that $X$ is a normal, $\kappa$-complete ultrafilter on $\kappa$ in $L[X]$.

In the next step we take $L[U]$ as the ground model and collapse the cardinal $\kappa$ to $\omega_1$ using
the Levy collapse. More particularly we let $\forceQ$ consist of functions
$p$ on subsets of $\kappa \times \omega$ such that
\begin{enumerate}
 \item $|dom(p)|$ is finite and
 \item $p(\alpha, n ) < \alpha$ for every $(\alpha, n) \in dom(p)$.
\end{enumerate}
It is wellknown that the Levy collapse
has the $\kappa$-cc, hence if $G_0$ denotes the generic filter for 
$\forceQ$ we obtain that $L[U][G_0] \models \kappa = \aleph_1$.
As a consequence of the $\kappa$-cc of $\forceQ$, every stationary subset of $\omega_1$ in $L[U][G_0]$
will contain an $S' \subset \kappa$ which is a stationary element of $L[U]$.
The Levy collapse is a weakly homogeneous forcing, which yields that every ordinal-definable
stationary, costationary $S \subset \omega_1$ is in fact an elment of $L[U]$ already.
Indeed if there is a formula $\varphi(x, \beta_0,..., \beta_n)$ such that $\alpha \in S$ if and only if $L[U][G_0] \models \varphi(\alpha, \beta_0,...,\beta_n)$
then the latter is equivalent to $L[U] \models \Vdash_{\forceQ} \varphi(\alpha, \beta_0,...,\beta_n)$ witnessing the ordinal-definability
of $S$ over $L[U]$ already.
Note however that $S$ will be a stationary, costationary subset of $\kappa$ in $L[U]$ as we have not 
collapsed $\kappa$ in $L[U]$.

We are now in the position to define the final iteration. Let $L[U][G_0]$ be the 
ground model. In $L[U][G_0]$, list the ordinal-definable, stationary, costationary subsets of $\omega_1$, write the list as $(S_i \, : \, i < \omega_2)$.
We can assume that this enumeration is ordinal definable over $L[U][G_0]$ as every ordinal-definable, stationary, costationary $S \subset \omega_1$
in $L[U][G_0]$ is, by the homogeneity of the Levy collapse already an element of 
$L[U]$ and hence we can use the canonical wellorder $<_{L[U]}$, which is ordinal definable over $L[U][G_0]$, to wellorder the $S_i$'s.
Let $\forceR:= (\forceR_{\alpha} \, : \, \alpha < \omega_2)$ be an $\omega_2$-length product defined inductively as follows:
suppose we are at stage $\alpha < \omega_2$, and $\forceR_{\alpha}$ has already been defined. Suppose
further that $S_{\alpha}$ is an ordinal-definable, stationary, costationary subset of
$\omega_1$ in $L[U][G_0]$.
Note here again that even though $\kappa$ has been collapsed to $\omega_1$ in $L[U][G_0]$,
by the homogeneity of the Levy collapse, if $S_{\alpha}$ is an ordinal-definable set in 
$L[U][G_0]$ it is also an element of $L[U]$. As $U$ is ordinal-definable over $L[U][G_0]$, 
the question whether $S_{\alpha}$ is an element of $U$ is therefore definable in 
$L[U][G_0]$, hence we can define the $\alpha$-th forcing of the product 
as follows:
\begin{itemize}
 \item If $S_{\alpha}$ is stationary in $L[U][G_0]$ and an element of $U$ we force with the club shooting 
forcing $\forceR_{\alpha} :=\forceP_{S_{\alpha}}$ which shoots a club set through $S_{\alpha}$.
 \item If $S_{\alpha}$ is not an element of $U$ we do nothing. 
\end{itemize}
We use countable support for the product, which is ordinal definable over $L[U][G_0]$.
Let $G_1$ be a generic filter for the product $(\forceR_{\alpha} \, : \, \alpha < \omega_2)$.
To finish the proof we need to show two things, namely that
$\omega_1^{L[U][G_0][G_1]} = \omega_1^{L[U][G_0]}$
and that in $L[U][G_0][G_1]$ there are no ordinal-definable, stationary, costationary subsets of $\omega_1$.

\begin{Claim}
 The product $(\forceR_{\alpha} \, : \, \alpha < \omega_2)$ is $\omega$-distributive
 and therefore preserves $\aleph_1$.
 \end{Claim}
\begin{proof}
 This is proved by induction on $\alpha < \omega_2$. First note that for a stationary set $S \subset \omega_1$, 
 club shooting itself is an $\omega$-distributive forcing which immediately gives the successor case.
 For limit ordinals $\alpha$, note that $\forceR_{\alpha}$ is the countably supported
 product of certain club shooting forcings $\prod_{\beta< \alpha} \forceP_{S_{\beta}}$. 
 We can assume without loss of generality that $\alpha = \omega_1$ 
 as we can pick a bijection $f : \alpha \rightarrow \omega_1$
 and use the fact that $\prod_{\beta < \alpha} \forceP_{S_{\beta}}$ is isomorphic to
 $\prod_{f(\beta) < f(\alpha)} \forceP_{S_{f(\beta)}}$.
 By definition, we only used a club shooting forcing $\forceP_{S_{\beta}}$ for an $S_{\beta}, \beta < \omega_1$ 
 if $S_{\beta}$ was an element of the ultrafilter $U$. As $U$ is normal and hence closed under diagonal intersections,
 $\Delta_{\beta < \omega_1} S_{\beta}$ is an element of $U$ again. Moreover elements of $U$ are stationary subsets of $\kappa$ in 
 $L[U]$, and by the $\kappa$-cc of the Levy collapse they remain stationary in $L[U][G_0]$, so $\Delta_{\beta< \omega_1} S_{\beta}$ is 
 a stationary subset of $\omega_1$ in $L[U][G_0]$.
 Let $\dot{r}$ be the $\forceR_{\omega_1}$-name of a real in $L[U][G_0]$. We shall show that there is a condition
 $q \in \forceR_{\omega_1}$ and a real $r \in L[U][G_0]$ such that $q \Vdash \dot{r}=r$.
 Fix a countable model $M \prec H_{\theta}$ for $\theta $ sufficiently large and which contains $\dot{r}$.
 As $\Delta_{\beta < \omega_1} S_{\beta}$ is stationary in $L[U][G_0]$ we can assume that $M \cap \omega_1 \in \Delta_{\beta < \omega_1} S_{\beta}$.
 We list the dense subsets $(D_n \, : \, n \in \omega)$ of $\forceR_{\omega_1}$ which are in $M$
 and recursively construct a descending sequence of elements of $\forceR_{\omega_1} \cap M$, $(p_n \, : \, n \in \omega)$,
 such that $p_n \in D_n$ and $p_{n+1}$ decides the value of $\dot{r} (n)$ and such that for every
 coordinate $i \in \omega_1$ such that there is an $n \in \omega$ for which $p_n (i) \ne 1$, the 
 value $sup_{n \in \omega} p_n (i) \cap \omega_1 = M \cap \omega_1$.
 We claim that the lower bound for the sequence $(p_n \, : \, n \in \omega)$, taken coordinatewise,
 i.e. $q:= \prod_{\beta \in \omega_1} \bigcup_{n \in \omega} p_n (\beta)$ 
 is a condition in $\forceR_{\omega_1}$. Indeed every nontrivial coordinate $i$ of a
 condition $p_n$ in the descending sequence is a countable, closed subset of the ordinal-definable, stationary, costationary 
 set $S_i$. Thus it is sufficient to show that for every coordinate $i$ in the support of $q$, the set $\bigcup_{n \in \omega} p_n(i)$
 is a closed subset of $S_i$, which boils down to ensure that $sup_{n \in \omega} (p_n (i) \cap \omega_1)$ is an element of 
 $S_i$. But $(p_n \, : \, n \in \omega) \subset M$ and $M \cap \omega_1 \in \Delta_{\beta < \omega_1} S_{\beta}$, and as $supp(q) \subset M \cap \omega_1$
 we know that for every $i \in supp(q)$ it must hold that
 $sup_{n \in \omega} (p_n (i) \cap \omega_1) = M \cap \omega_1$ is an element of $S_i$. So $q$ is a condition in $\forceR_{\omega_1}$ which
 witnesses that the name of a real $\dot{r}$ is already an element of $L[U][G_0]$.
 
 To show that the full product $\forceR=\prod_{\alpha < \omega_2} \forceR_{\alpha}$ is $\omega$-distributive
 note that in our ground model $L[U][G_0]$ the club shooting forcings have size $\aleph_1$. As a consequence
 every $\forceR$-name of a real $\dot{r}$ is in fact a name in some initial segment $\prod_{\beta < \alpha} \forceR_{\beta}$ 
 for $\alpha < \omega_2$. But we have seen already that $\prod_{\beta < \alpha} \forceR_{\beta}$ is $\omega$-distributive, so we are finished.

\end{proof}

What is left is to show that in $L[U][G_0][G_1]$ there are no ordinal-definable, stationary, costationary
subsets of $\omega_1$. Assume for a contradiction that $S$ would be such a set in $L[U][G_0][G_1]$.
Then as $\forceR$ is the product of homogeneous forcings it is homogeneous itself and so
$S$ is an element of $L[U][G_0]$. The forcing $\forceR$ is ordinal-definable over $L[U][G_0]$ as
the sequence of the ordinal-definable, stationary, costationary subsets of $\omega_1$ in $L[U][G_0]$
is itself ordinal definable there, using the canonical (ordinal-definable) wellorder of $L[U]$.
Indeed, working in $L[U][G_0]$ we can first define $U$ to be the unique $\kappa$-closed and normal ultrafilter $X$ on $\kappa$
which is an element of $L[X]$, and having defined $U$ we can also define the canonical $L[U]$-wellorder in 
an ordinal-definable way, which we use to define (with $\omega_1= \kappa$ as the parameter) the sequence of ordinal-definable, stationary, costationary 
subsets of $\omega_1$ in $L[U][G_0][G_1]$. But as $\forceR$ is a homgeneous, ordinal-definable forcing over $L[U][G_0]$
we see that any ordinal definable set $S$ over $L[U][G_0][G_1]$ is in fact ordinal definable over $L[U][G_0]$:
if $\varphi(x, \alpha,..)$ is such that $\beta \in S$ if and only if $L[U][G_0][G_1] \models \varphi(\beta, \alpha,...)$
then the latter is equivalent to $L[U][G_0] \models 1\Vdash_{\forceR} \varphi (\check{\beta}, \check{\alpha},...)$
and the relation $\Vdash_{\forceR}$ and the weakest condition 1 are ordinal definable.
To summarize, an ordinal-definable, stationary, costationary element $S$ of $L[U][G_0][G_1]$ is
ordinal-definable, stationary, costationary already in $L[U][G_0]$. But we have killed the stationarity or costationarity of every 
such set with the forcing $\forceR$. So there is no $S$ which is ordinal-definable and stationary and costationary in $L[U][G_0][G_1]$ as desired.

\end{proof}

\end{document}